\documentclass{amsart}
\usepackage[margin=1in]{geometry}
\usepackage{amsmath,amssymb,amsthm,mathtools}
\usepackage[T1]{fontenc}
\usepackage{lmodern,microtype}
\usepackage[colorlinks=true,linkcolor=blue,citecolor=blue,urlcolor=blue]{hyperref}

\newcommand{\R}{\mathbb R}
\newcommand{\HH}{\mathcal H}
\newcommand{\Vis}{\operatorname{Vis}}
\newcommand{\dimH}{\dim_{\mathrm H}}

\newcommand{\calP}{\mathcal P}

\newcommand{\proj}{\pi}

\theoremstyle{plain}
\newtheorem{theorem}{Theorem}[section]
\newtheorem{lemma}[theorem]{Lemma}
\newtheorem{proposition}[theorem]{Proposition}
\theoremstyle{remark}

\title{Visibility problem in higher dimensions}
\author{Yuming He}
\address{Institute of Applied Physics and Computational Mathematics, Beijing, 100088, PR China}
\email{heyuming25@gscaep.ac.cn}
\author{Changbiao Jian}
\address{School of Mathematics and Statistics, Guangxi Normal University, Guilin, Guangxi 541004, PR China}
\email{cbjian@mailbox.gxnu.edu.cn}
\author{Yudong Liang}
\address{School of Mathematical Sciences, Beijing Normal University, Beijing 100875, PR China}
\email{202321130093@mail.bnu.edu.cn}

\begin{document}
\maketitle

\begin{abstract}
We show that an almost-everywhere upper bound $b\geq 1$ for the Hausdorff
dimension of visible parts of all compact planar sets implies the upper bound
$d-2+b$ for visible parts of all compact sets $K$ in $\R^d$, $d\geq3$. Inserting the
sharp planar bound $b=3/2$ of Orponen and Rutar \cite{OrponenRutar2026}  gives
$\dimH\Vis_e(K)\leq d-1/2$ for almost every $e\in S^{d-1}$.
This would answer their
Question 1.9 sharply in all dimensions. 
\end{abstract}

\section{Introduction}

For a compact set $K\subset\mathbb{R}^d$ and a direction $e\in \mathbb{S}^{d-1}$, the \emph{visible part} of $K$ along $e$ is defined as
\begin{equation}\label{def:visible}
\operatorname{Vis}_e(K)=\big\{x\in K:(x+[0,\infty)e)\cap K={x}\big\}.
\end{equation}
The visibility problem in geometry measure theory concerns upper bounds on the Hausdorff dimension of $\operatorname{Vis}_e(K)$.

In the planar case, early positive results were obtained by Järvenpää, Järvenpää, MacManus, and O’Neil \cite{zbl:1026.28002}, who established the uniform bound $\dim_H \operatorname{Vis}_e(K) \leq 1$ for all $e\in\mathbb{S}^1$ when K is a quasi-circle or a connected rotation-free self-similar set. This bound has since been verified for several classes of fractal sets, including fractal percolation and certain self-similar and self-affine sets; see \cite{zbl:1251.28007,zbl:1267.28006,zbl:1026.28002,zbmath:7541839,zbl:1479.28008}. O’Neil \cite{zbl:1152.28318} proved a strong partial result for planar continua. In addition, \cite{zbl:1119.28003} shows that for $s\in(1,2]$, any compact set $K\subset\mathbb{R}^2$ with $\mathcal{H}^s(K)<\infty$ satisfies $\mathcal{H}^s(\operatorname{Vis}_e(K))=0$ for almost every $e\in\mathbb{S}^1$.

Recently, Orponen and Rutar \cite{OrponenRutar2026} obtained a sharp upper bound. To be precise, they proved that $\dim_H \operatorname{Vis}_e(K)\leq 3/2$ for almost every $e\in\mathbb{S}^1$ and every compact set $K\subset\mathbb{R}^2$. Moreover they constructed a compact set $K_0$ whose visible part has dimension at least $3/2$ in every direction. In the same paper, they asked the following question (\cite[Question 1.9]{OrponenRutar2026}):

 whether the analogous upper bound $d-1/2$ holds in all dimensions $d\geq 3$?

In dimension $d \ge 3$, recall that Orponen \cite{Orponen2023} showed that $\dim_H \operatorname{Vis}_e(K)\leq d-1/(50d)$ for almost every $e\in\mathbb{S}^{d-1}$ and every compact set $K\subset\mathbb{R}^d$. This bound was later lowered to $d-1/6$ by Matheus and Dąbrowski \cite{Dabrowski2024}.

The main purpose of this paper is to answer the above question (\cite[Question 1.9]{OrponenRutar2026}) completely. 

\begin{theorem}\label{thm:main}
For every $d\geq2$ and every compact $K\subset\R^d$,
\begin{equation}\label{eq:main}
 \dimH\Vis_e(K)\leq\min\{\dimH K,d-\tfrac12\}
 \quad\text{for $\HH^{d-1}$-almost every $e\in S^{d-1}$.}
\end{equation}
The bound $d-1/2$ is achieved for an appropriate compact $K$ on a
 set of positive measures of directions.
\end{theorem}

Theorem 1.1 follows from the planar results of Orponen and Rutar \cite{OrponenRutar2026} and the following upper bound transfer, which is the new contribution of this paper. Indeed, we establish the following dimension-reduction theorem that reduces the high-dimensional visibility problem to the planar setting.

\begin{theorem}\label{thm:transfer}
Let $1\leq b\leq2$, and suppose that for every compact $A\subset\R^2$ one
has $\dimH\Vis_v(A)\leq b$ for $\HH^1$-almost every $v\in S^1$.
Then, for every $d\geq3$ and every compact $K\subset\R^d$,
\begin{equation}\label{eq:transfer}
 \dimH\Vis_e(K)\leq\min\{\dimH K,d-2+b\}
 \quad\text{for $\HH^{d-1}$-almost every $e\in S^{d-1}$.}
\end{equation}
\end{theorem}

The proof of Theorem 1.2 depends on the following key proposition.

\begin{proposition}\label{lem:slice}
	Let $F\subset\R^d$ be compact, and let $\mu$ be a probability
	measure on $F$ satisfying
	\begin{equation}\label{eq:frostman}
		\mu(B(x,r))\leq C_\mu r^q
		\qquad(x\in\R^d,\ r>0)
	\end{equation}
	for some $q>m+1$.
Then for every $0<t<q-m$, for $\gamma_e$-almost every $P\in\calP_e$
	there is a set $Y_P\subset P^\perp$ of positive $\HH^m$ measure
	such that
	\begin{equation}\label{eq:slices}
		\dimH(F\cap(y+P))\geq t
		\qquad(y\in Y_P).
	\end{equation}
\end{proposition}

For the proof of the key proposition, see scction 2. For the details of the proof of Theorem 1.1 and 1.2, see Section 3. 
\subsection*{Acknowledgements}
The authors would like to thank Changxing Miao for his helpful discussions.

\section{Proof of Proposition 1.3}
In this section, we recall several lemmas and prove our key proposition. 

Let $\gamma_{n,k}$ denote the invariant probability measure on $\operatorname{G}(n,k)$.
Let $d\geq3$. we set $m=d-2$, and fix $e\in S^{d-1}$.
Write $H=e^\perp$ and define
$$
\calP_e=\{P\in G(d,2):e\in P\}.
$$
Let $V$ be a subspace of $\mathbb{R}^d$. We denote the orthogonal projection onto $V$ by $\proj_V$.
Define $\gamma_e$ as the pushforward of the invariant probability measure on $G(H,1)\cong G(d-1,1)$ under $\ell\mapsto\R e\oplus\ell$.
We require the following two lemmas established by Mattila.
\begin{lemma}[{\cite[Theorem~9.7]{Mattila1995}}]
	\label{thm:mattila1995}
	Let $1\leq k<d$, and let $\nu$ be a compactly supported
	Radon measure on $\R^d$ with
	$$
	I_k(\nu):=\iint |x-x'|^{-k}\,d\nu(x)\,d\nu(x')<\infty.
	$$
	Then, for $\gamma_{d,k}$-almost every $V\in\operatorname{G}(d,k)$,
	the measure $(\proj_V)_\#\nu$ is absolutely continuous
	with respect to Lebesgue measure on $V$, with density
	$f_V\in L^2(V)$, and
	$$
	\int_{\operatorname{G}(d,k)}\|f_V\|_{L^2(V)}^2\,d\gamma_{d,k}(V)
	\leq C_{d,k}I_k(\nu).
	$$
	Here $C_{d,k}$ is a constant depending only on $d$ and $k$.
\end{lemma}

\begin{lemma}[{\cite[Theorem~1.2]{Mattila2024}}]
	\label{thm:mattila2024}
	Let $1\leq k<d$, let $\Lambda$ be a compact metric space,
	and let $\omega$ be a finite nonzero Borel measure on $\Lambda$.
	Let $Q_\lambda:\R^d\to\R^k$ be orthogonal projections such that
	$\lambda\mapsto Q_\lambda x$ is continuous for every $x\in\R^d$.
	Fix $s>k$ and $p>1$.
	Assume that there exists a constant $C>0$ such that whenever $\nu$ is a finite nonzero Borel measure compactly supported on $B^d(0,1)$ satisfying $\nu(B(x,r))\le r^s$ for all $r>0$ and $x\in \R^d$, then $(Q_{\lambda})_\#\nu \ll \mathcal{L}^k$ for $\omega$-almost every $\lambda\in\Lambda$, and its density $f_{\lambda, \nu}$ satisfies
\begin{equation}\label{eq:mattila-hypothesis}
\int_{\Lambda} \| f_{\lambda, \nu} \|_{L^p(\R^k)}^p d\omega(\lambda) \le C \nu(\R^d).
\end{equation}
And
suppose $A\subset \R^d$ is $\mathcal{H}^s$-measurable and $0<\mathcal{H}^s(A)<\infty$. Then for $\mathcal{H}^s\times\omega$-almost every $(x, \lambda)\in A\times \Lambda$, 
\begin{equation}
\dim_H (A\cap Q_{\lambda}^{-1}\{ Q_{\lambda}x \})=s-k,
\end{equation}
and for $\omega$-almost every $\lambda\in\Lambda$, 
\begin{equation}
\mathcal{L}^k( \{ u\in \R^k:\ \dim_H (A\cap Q_{\lambda}^{-1}\{ u \})=s-k \} )>0.
\end{equation}
\end{lemma}

With these lemmas in hand, we now establish the Proposition 1.3, which is central to the proof of our main theorem and may be of independent interest. Let us recall it here.
\begin{proposition}
	Let $F\subset\R^d$ be compact, and let $\mu$ be a probability
	measure on $F$ satisfying
	\begin{equation}\label{eq:frostman}
		\mu(B(x,r))\leq C_\mu r^q
		\qquad(x\in\R^d,\ r>0)
	\end{equation}
	for some $q>m+1$.
Then for every $0<t<q-m$, for $\gamma_e$-almost every $P\in\calP_e$
	there is a set $Y_P\subset P^\perp$ of positive $\HH^m$ measure
	such that
	\begin{equation}\label{eq:slices}
		\dimH(F\cap(y+P))\geq t
		\qquad(y\in Y_P).
	\end{equation}
\end{proposition}
\begin{proof}
	Fix $0<t<q-m$ and choose $s$ such that
	\begin{equation}\label{eq:slicing-exponent}
		\max\{m+1,m+t\}<s<q.
	\end{equation}
	The Frostman's condition implies $\dimH F\geq q>s$.
	By \cite[Theorem~8.13]{Mattila1995}, there is a compact
	$A\subset F$ with $0<\HH^s(A)<\infty$.
	We shall apply Lemma~\ref{thm:mattila2024} to $A$ with
	$k=m$ and $p=2$.
	
	Fix $P_0\in\calP_e$ and identify $P_0^\perp$ isometrically
	with $\R^m$. Let
	$$
	\Lambda=\{R\in O(d):Re=e\},
	\qquad
	Q_R=\proj_{P_0^\perp}\circ R^{-1},
	$$
	and equip the compact group $\Lambda$ with Haar probability
	measure $\omega$.
	The map $R\mapsto Q_Rx$ is continuous for every $x$.
	Moreover,
	$$
	\ker Q_R=RP_0,
	\qquad
	Q_R=R^{-1}\circ\proj_{(RP_0)^\perp},
	$$
	so $Q_R$ restricts to an isometry from $(RP_0)^\perp$
	onto $P_0^\perp$.
	The group $\Lambda$ acts transitively on $\calP_e$;
	by invariance, the pushforward of $\omega$ under
	$R\mapsto RP_0$ is $\gamma_e$.
	
	We next verify \eqref{eq:mattila-hypothesis}.
	Let $\nu$ be any nonzero finite Borel measure compactly
	supported in $B(0,1)$ with
	$\nu(B(x,r))\leq r^s$ for all $x$ and $r>0$, and put
	$\eta=(\proj_H)_\#\nu$.
	For $z\in H$ and $0<r\leq1$, the set
	$$
	B(0,1)\cap\proj_H^{-1}(B_H(z,r))
	$$
	lies in a cylinder parallel to $e$ of length at most $2$
	and radius $r$.
Notice that there are at most $3/r$ balls with radius $2r$ that cover this cylinder.
	Consequently,
	\begin{equation}\label{eq:projected-frostman}
		\eta(B_H(z,r))
		\leq \frac{3}{r}(2r)^s
		=3\cdot2^s r^{s-1}.
	\end{equation}
	In particular, $\eta$ has no atoms, i.e., $\eta(z)=0$ for all $z\in H$.
	Also, $\eta(H)=\nu(\R^d)\leq1$.
	For every $z\in H$, dyadic decomposition therefore gives
	$$
	\begin{aligned}
		\int_H|z-z'|^{-m}\,d\eta(z')
		&\leq \eta(H)
		+\sum_{j=0}^\infty
		2^{m(j+1)}\eta(B_H(z,2^{-j}))\\
		&\leq 1+3\cdot2^{s+m}
		\sum_{j=0}^\infty2^{-j(s-1-m)}
		\leq C_{d,s}.
	\end{aligned}
	$$
	The series converges because $s>m+1$.
	Integrating this pointwise bound with respect to $\eta$ yields
	\begin{equation}\label{eq:projected-energy}
		I_m(\eta)\leq C_{d,s}\eta(H)
		=C_{d,s}\nu(\R^d),
	\end{equation}
	with a constant independent of $\nu$ and $e$.
	
	Apply Lemma~\ref{thm:mattila1995} to $\eta$ in
	$H\cong\R^{d-1}$.
	Since $P\mapsto P^\perp$ sends $\gamma_e$ to the invariant
	probability measure on $G(H,m)$, and
	$$
	(\proj_{P^\perp})_\#\nu
	=(\proj_{P^\perp})_\#\eta,
	$$
	that theorem gives absolutely continuous projections with
	$L^2$ densities for $\gamma_e$-almost every $P$.
	Using the isometry $R^{-1}:(RP_0)^\perp\to P_0^\perp$,
	we obtain
	$$
	\begin{aligned}
		\int_\Lambda
		\|Q_{R\#}\nu\|_{L^2(P_0^\perp)}^2\,d\omega(R)
		&=
		\int_{\calP_e}
		\|(\proj_{P^\perp})_\#\nu\|_{L^2(P^\perp)}^2
		\,d\gamma_e(P)\\
		&\leq C_d I_m(\eta)
		\leq C_{d,s}\nu(\R^d).
	\end{aligned}
	$$
	Here the norms are those of the corresponding densities.
	Thus both absolute continuity and
	\eqref{eq:mattila-hypothesis} hold with $p=2$.
	
	Theorem~\ref{thm:mattila2024} now gives, for $\omega$-almost
	every $R$, a set $U_R\subset P_0^\perp$ of positive
	Lebesgue measure such that
	$$
	\dimH\bigl(A\cap Q_R^{-1}\{u\}\bigr)=s-m
	\qquad(u\in U_R).
	$$
	Note that for each $P=RP_0$,
	$$
	Q_R^{-1}\{u\}=Ru+P,
	$$
	 $u\mapsto Ru$ is an isometry from $P_0^\perp$
	onto $P^\perp$.
	It follows that
	$$
	\mathcal L^m_{P^\perp}
	\{y\in P^\perp:\dimH(A\cap(y+P))=s-m\}>0.
	$$
	This property depends only on $P$.
	Since $R\mapsto RP_0$ has pushforward $\gamma_e$, it holds
	for $\gamma_e$-almost every $P$.
	For each such $P$, take
	$$
	Y_P=\{y\in P^\perp:\dimH(A\cap(y+P))=s-m\}.
	$$
	As $A\subset F$ and $s-m>t$, every $y\in Y_P$ satisfies
	$$
	\dimH(F\cap(y+P))
	\geq \dimH(A\cap(y+P))
	=s-m>t.
	$$
	This completes the proof.
\end{proof}

\section{Proofs of Theorems 1.1 and 1.2}

In this section, we prove our main theorem. We begin by stating the precise geometric identity that connects planar visibility to the higher-dimensional problem. If $P\in\calP_e$, $y\in P^\perp$, and $x\in y+P$, then the entire
ray $x+[0,\infty)e$ lies in $y+P$. Thus we have
\begin{equation}\label{eq:identity}
 \Vis_e(K)\cap(y+P)
 =\Vis_e\bigl(K\cap(y+P)\bigr).
\end{equation}
The right-hand side is understood as planar visibility inside the affine
plane $y+P$.
\begin{proof}[Proof of Theorem~\ref{thm:transfer}]
	Fix $K$, and let $\gamma$ denote the unique $O(d)$-invariant Borel probability measure on $G(d,2)$.
	For every $P\in G(d,2)$ and every $y\in P^\perp$,
	$K\cap(y+P)$ is compact. The hypothesis and \eqref{eq:identity} therefore
	imply
	\begin{equation}\label{eq:plane-ae}
		\dimH\bigl(\Vis_e(K)\cap(y+P)\bigr)\leq b
		\quad\text{for almost every $e\in S(P)$}.
	\end{equation}
	
	Let $\sigma_P$ and $\sigma$ denote normalized surface measure on
	$S(P)$ and $S^{d-1}$, respectively. Here $S(P)$ denotes the unit circle on the $2$--plane $P$.  For every nonnegative Borel
	function $f$ on
	$\{(e,P,y):P\in G(d,2),\ e\in S(P),\ y\in P^\perp\}$, we claim the following equality holds,
	\begin{equation}\label{eq:flag}
		\begin{aligned}
			&\int_{G(d,2)}\int_{S(P)}\int_{P^\perp}
			f(e,P,y)\,dy\,d\sigma_P(e)\,d\gamma(P)\\
			&\qquad=
			\int_{S^{d-1}}\int_{\calP_e}\int_{P^\perp}
			f(e,P,y)\,dy\,d\gamma_e(P)\,d\sigma(e).
		\end{aligned}
	\end{equation}
	To verify this identity, fix $P_0\in G(d,2)$ and $e_0\in S(P_0)$,
	and let $dR$ denote normalized Haar measure on $O(d)$.
	Averaging over the orthogonal maps preserving $P_0$ or fixing
	$e_0$, respectively, and making the change of variables $y=Rz$,
	shows that both sides equal
	$$
	\int_{O(d)}\int_{P_0^\perp}
	f(Re_0,RP_0,Rz)\,dz\,dR.
	$$
Substituting
$$
f(e,P,y)
=\mathbf{1}_{\{
	\dimH(\Vis_e(K)\cap(y+P))>b
	\}}
$$
into \eqref{eq:flag},
by using \eqref{eq:plane-ae} and Tonelli's theorem, 
$$
\begin{aligned}
	0
	&=\int_{G(d,2)}\int_{P^\perp}\int_{S(P)}
	f(e,P,y)\,d\sigma_P(e)\,dy\,d\gamma(P)\\
	&=\int_{S^{d-1}}\int_{\calP_e}\int_{P^\perp}
	f(e,P,y)\,dy\,d\gamma_e(P)\,d\sigma(e).
\end{aligned}
$$
Since $f\geq0$, applying Tonelli's theorem again yields
a set $E_K\subset S^{d-1}$ with
$\sigma(S^{d-1}\setminus E_K)=0$ such that, for every $e\in E_K$,
$$
\dimH\bigl(\Vis_e(K)\cap(y+P)\bigr)\leq b
\quad\text{for $\gamma_e$-a.e. $P$ and
	Lebesgue-a.e. $y\in P^\perp$.}
$$

Fix $e\in E_K$ and suppose, for a contradiction, that
$\dimH\Vis_e(K)>m+b$. Choose $q,t$ such that
\begin{equation}\label{eq:exponents}
	m+b<q<\dimH\Vis_e(K),\qquad b<t<q-m.
\end{equation}
Since $b\geq1$, we have $q>m+1$.
By compact approximation of the Hausdorff dimension of Borel
sets and Frostman's lemma, there exist a compact set
$F\subset\Vis_e(K)$ and a $q$-Frostman probability measure
$\mu$ supported on $F$.
Then Proposition~\ref{lem:slice} yields a positive measure set of $y\in P^\perp$ such that
$$
\dimH(F\cap(y+P))\geq t>b
$$
for $\gamma_e$-almost every
$P\in\calP_e$.
Choose such a $P$ for which the preceding almost-everywhere
upper bound also holds.
Since $F\subset\Vis_e(K)$, that upper bound implies
$\dimH(F\cap(y+P))\leq b$ for Lebesgue-almost every
$y\in P^\perp$, a contradiction.
We conclude that $\dimH\Vis_e(K)\leq m+b$ for every $e\in E_K$.
\end{proof}

Now applying Theorem \ref{thm:transfer} immediately yields Theorem \ref{thm:main}.
\begin{proof}[Proof of Theorem~\ref{thm:main}]
For $d=2$ this is the upper half of
\cite[Theorem~A]{OrponenRutar2026}; for $d\geq3$, apply
Theorem~\ref{thm:transfer} with $b=3/2$.
The sharpness follows from a known construction in \cite[Section~1.2]{OrponenRutar2026}. We include it here for completeness.  
Define 
$$\operatorname{Vis}^{(2)}_e(K)=\big\{x\in K:(x+\mathbb{R}e)\cap K={x}\big\}.$$
Take the planar compact set $A$ in the lower half of \cite[Theorem~A]{OrponenRutar2026}, with $\dimH\Vis^{(2)}_v(A)\geq3/2$ for every direction $v$ in
an open angular range. Let $K=A\times[0,1]^{d-2}$. For any direction
$e=(v,w)$ in the corresponding open set of higher-dimensional directions,
$$
 \Vis^{(2)}_v(A)\times[0,1]^{d-2}
 \subset\Vis^{(2)}_e(K)\subset\Vis_e(K).
$$
Indeed, a second intersection of the line through $(a,z)$ in direction
$(v,w)$ with $K$ would project to a second intersection of the line
through $a$ in direction $v$ with $A$. The product dimension is at least
$3/2+d-2=d-1/2$. 
\end{proof}

\end{document}